\documentclass{amsart}
\usepackage{amsmath,amsthm,amssymb}
\usepackage{caption}
\usepackage{subcaption}
\usepackage{color}
\usepackage{hyperref}
\usepackage{url}
\usepackage{graphicx, float}

\newcommand{\bburl}[1]{\textcolor{blue}{\url{#1}}}

\newtheorem{thm}{Theorem}[section]

\newtheorem{lem}[thm]{Lemma}

\theoremstyle{definition}
\newtheorem{exa}[thm]{Example}
\theoremstyle{definition}
\newtheorem{defi}[thm]{Definition}
\theoremstyle{remark}

\newcommand\be{\begin{equation}}
\newcommand\ee{\end{equation}}
\numberwithin{equation}{section}

\author[Y. Gaur]{Yashaswika Gaur}
\email{\textcolor{blue}{\href{mailto:yashaswikagaur@gmail.com}
{yashaswikagaur@gmail.com}}}

\author[T.A. Wong]{Tian An Wong}
\email{\textcolor{blue}{\href{mailto:t.wong@columbia.edu}
{t.wong@columbia.edu}}}



\date{\today}

\begin{document}

\title{Lattice points in vector-dilated quadratic irrational polytopes}

\maketitle

\begin{abstract}
We study the Ehrhart theory of quadratic irrational polytopes that undergo vector dilations. That is, for a given polytope with vertices in $\mathbb{Q}(\sqrt{D})$, and a different dilation factor for each facet, we show that the leading term of the lattice-point count behaves similar to an Ehrhart polynomial, generalizing previous work of Borda on scalar dilations of quadratic irrational polytopes. As a result, a form of the Ehrhart-Macdonald reciprocity law is obtained for the leading term.
\end{abstract}

\section{Introduction}
For a simple polygon with points with integer coordinates such that all its vertices are integers, Pick's theorem provides the following formula:
\be
 A = I+\frac{1}{2}B -1,
\ee
where $A$ is the area of the polygon, $I$ the number of lattice points in the interior, and $B$ the number of lattice points on the boundary of the polygon.

A rational polytope is defined by a system of inequalities determined by $A \in \text{Mat}_{m\times  n}(\mathbb Z)$, $ b \in \mathbb{Z}^m$, $m\ge n$ as:
\be
\label{polyt}
P = \{x \in \mathbb{R}^n: Ax \leq b\},
\ee
and a $t$-fold dilation, for a fixed $t>0$ is defined to be the set of points 
\be
tP = \{ tx : x\in P\}.
\ee
For a given $t$-dilated polytope, there exists a polynomial called the Ehrhart polynomial \cite{Eh}, that measures the volume of the polytope. This is a higher dimensional generalization of Pick's theorem in the Euclidean plane. 

A quasipolynomial, also called a pseudopolynomial, is of the form 
\be
c_n(t) t^n + \dots + c_1(t) t + c_0(t),
\ee
where $c_0, \dots ,c_n$ are periodic functions in the integer variable $t$. The key difference between a polynomial and a quasipolynomial is that the coefficients of a polynomial belong to a ring, but the coefficients of quasipolynomials are periodic functions. Then Erhart's theorem is as follows:

\begin{thm}[Ehrhart]\label{Ehrhart}
Let $P$ be a rational polytope with a $t$-fold dilation. Then the Ehrhart polynomials of $P$, denoted $L(P,t)$ is a quasipolynomial in the integer variable $t$, where the leading term of $L(P,t)$ is equal to the volume of $P$. Moreover, if $P$ is a lattice polytope then $L(P,t)$ is a polynomial in $t$. 
\end{thm}

\noindent Much work has been done since on the lattice point problem in polytopes. Applications of such counting problems occur in pure areas of math such as toric Hilbert functions, Kostant’s partition function in representation theory as well as in the applied aspects of cryptography, integer programming, contingency tables. A detailed view of these can be found in \cite{DeL}. Counting problems in vector-dilated polytopes are specifically used in a particular kind of code in the field of coding theory, illustrated in \cite{Lis}. Because of their property of filling any kind of spaces of different dimensions, irrational polytopes are used in the development of sustainable nanosystems.

More recently, Beck has studied the lattice point problem in vector-dilated rational polytopes \cite{Be}, and Borda similarly for scalar-dilated quadratic irrational polytopes \cite{Be}, also referred to as algebraic cross-polytopes. In this paper, we combine the work of Beck and Borda, namely, focusing on polytopes that have vertices that lie in a fixed quadratic extension $\mathbb Q(\sqrt D)$, and different dilation factors for each facet, instead of a common dilation for the entire polytope. We refer to Theorem \ref{main} for the statement of the main result.

For any general polytope in $n$ dimension, we prove that its counting polynomial behaves like an Ehrhart polynomial and satisfies a form of the Ehrhart Macdonald reciprocity law, that is,
\be\label{rec}
f(-t) = (-1)^n f(t).
\ee
Our method is a proceeds by induction, taking into account the polynomials defined in \cite{Bo} for irrational polytopes.

In Sect. \ref{quad}, we define quadratic irrational polytopes and their vector dilations. In Sect. \ref{thm} we state the main theorem for vector dilations of quadratic irrational polytopes, and provide the proof in Sect. \ref{proof}.

\section{Quadratic irrational polytopes}
\label{quad}

We first introduce the class of quadratic irrational polytopes by extending the definition \eqref{polyt} such that the coefficients of the linear inequalities determined by $A$ and $b$ are allowed to irrational.  

\begin{defi}[Quadratic irrational polytope] Consider the real quadratic field, $ \mathbb{Q}(\sqrt D)$, where $D>1$ is a square free integer. A polytope $P$ with vertices in $ \mathbb{Q}(\sqrt D)^d$ containing the origin forms the quadratic irrational class of polytopes.
\end{defi}

\begin{exa}
Consider the lattice counting problem in a regular dodecahedron. In dimension $d=3$, not every polytope can be decomposed into orthogonal simplices.  We cannot embed a regular dodecahedron in $\mathbb{R}^3$ such that all its vertices are lattice points. The faces of a regular dodecahedron are regular pentagons and there does not exist a pentagon in $\mathbb{Z}^d$ for any $d$. A standard regular dodecahedron has 20 vertices,
\be
(\pm 1, \pm 1 , \pm 1), (0, \pm \phi^{-1},
\pm \phi),(\pm \phi, 0,  \pm \phi^{-1}), \phi^{-1}, \pm \phi,0 ),
\ee
where $ \phi = \frac{1+ \sqrt5}{2}$ is the golden ratio. This means that there is no Ehrhart polynomial for such a polytope that gives the precise number of lattice points in its integral dilates. However, it can be embedded in such a way that all its vertices lie in $ \mathbb{Q}(\sqrt5)^3$.
\end{exa}

Instead of dilating the polytope by a single factor, we allow different dilation factors for each facet, such that the dilated polytope is combinatorially equivalent to the original polytope. Recall that two polytopes are said to be combinatorially equivalent if their face lattices are isomorphic.

\begin{defi}[Vector-dilated polytopes] For a polytope defined in \eqref{polyt} and $\textbf{t} \in \mathbb{Z}^m$, define the vector-dilated polytope as
\be
P^{(t)} = \{x \in \mathbb{R}^n: Ax \leq \textbf{t}\}.
\ee
For vectors {\bf t} for which $P^{(t)}$ is combinatorially equivalent to $P=P^{(b)}$, we may define the number of lattice points in the interior and closure of $P^{(t)}$ to be
\be
i_P(t) = |P^{(t)^0} \cap \mathbb{Z}^n|,
\ee
and 
\be
j_P(t) = |P^{(t)} \cap \mathbb{Z}^n|.
\ee
\end{defi}

The condition of combinatorial equivalence is necessary since different dilation factors can easily cause distortion of the polytope. It also matters how we define the polytope in terms of half spaces.

\begin{exa}
Observe the following example, consider the system of inequalities:
\be
P= \{  x \in \mathbb{R}^n : Ax \leq b\},
\ee
where 
\be
 A =  \begin{bmatrix}
   1 & 1 \\
   -1 & 0 \\
   0 & -1 \\
  \end{bmatrix} ,\quad   
  b = \begin{bmatrix}
  \sqrt{2} \\
  0 \\
  0 \\
  \end{bmatrix},\quad   
  t = \begin {bmatrix}
      4 \\
      1 \\
      3 \\
  \end{bmatrix} .
  \ee
The new system of equations for $P$ becomes 
  \begin{center}
      $ x+y \leq \sqrt2 $\\
      $ -x \leq 0 $ \\
      $-y \leq 0 $, \\
  \end{center}
and on dilation with $t \in \mathbb{Z}^3 $,
  \begin{center}
      $ x+y \leq 4 $\\
      $ -x \leq 1 $ \\
      $-y \leq 3 $, \\
  \end{center}
which are the lines, 
\be
x+y \leq 4,\quad x \geq -11,\quad y \geq -3 .
\ee
The intersection of these half spaces gives the desired dilated triangle. 
\end{exa}
  
Ehrhart's theorem (Theorem \ref{Ehrhart}) can be extended for a rational polytope considering vector dilations as done in \cite{Be}. Using induction on the dimension, Beck proves the following:
\begin{thm}[Beck]\label{Beck}
 Let $S$ be an $n$-dimensional rational simplex. Then $i_S(t$) and $j_S(t)$ are quasipolynomials in $t \in \mathbb{Z}^{n+1}$, satisfying $i_S(-t) = (-1)^n j_S(t)$.
\end{thm}
The proof of Beck's theorem goes roughly as follows: consider an $n$ dimensional rational simplex written as $n$ linear inequalities. Looking at how a single vertex undergoes dilation, one easily obtains a bound in one dimension. Assuming the induction hypothesis holds for all polytopes in $n-1$ dimension say $Q$, the polynomials of  $P$ can be written as a sum of the quasipolynomial of $Q$ over the bounds obtained previously. This implies that $P$ associates with quasi-polynomials and enables one to prove the Ehrhart-Macdonald reciprocity law.

In the following, we shall call a polynomial will be called {\em Ehrhart-like} or, say that it behaves like an Ehrhart polynomial if it satisfies the properties of an Ehrhart polynomial up to an error term.

\section{Statement of the main theorem}\label{thm}

Consider the two polytopes, for $d \geq 2$,  defined as
\be
\label{C}
C = \{ x \in \mathbb{R}^n : \dfrac{|x_1|}{a_1} + \dfrac{|x_2|}{a_2} + \dots + \dfrac{|x_d|}{a_n} \leq 1 \}    
\ee
and
\be
\label{S}
S = \{ x \in \mathbb{R}^n : \dfrac{x_1}{a_1} + \dfrac{x_2}{a_2} + \dots + \dfrac{x_d}{a_n} \leq 1 \}.
\ee
The vertices of $C$ are of the form $(0, \dots, \pm a_i, \dots ,0)$ and we assume that $ \frac{1}{a_1}, \frac{1}{a_2}, \dots , \frac{1}{a_d}$ are algebraic and linearly independent over $ \mathbb{Q}$. $S$, on the other hand always has the origin as a vertex and is a generalization of a right triangle in two dimensions. 

Another generalization of Theorem \ref{Ehrhart} has been done in \cite{Bo}. Borda uses Poisson summation formula to approximate the counting polynomial for an irrational polytope dilated by a scalar $t$ in all directions. For a polytope $C$ described in \eqref{C}, let $\chi_{tC}$ denote the characteristic function of $tC$ and 
\be
 \hat{\chi}_{tC} (y) = \int_{tC} e^{-2\pi i<x,y>}dx
 \ee
 denote its Fourier transform. The Cesaro mean of $tC$ is defined as \be
\text{Ces}(tC,N) = \dfrac{1}{N^d} \sum_{M \in [N-1]^d}  \sum_{m \in [-M_1,M_1] X \dots X [-M_d,M_d]} \hat{\chi}_{tC}(m). 
\ee
The Poisson summation formula leads to the estimate
\be 
|tC \cap \mathbb{Z}^d| =\text{Ces} (tC, N)+ O\left(1+ t^{d-1+ \epsilon} \sqrt{\dfrac{\log N}{N}}\right),
\ee
discussed in  \cite{Bo}. Now, $C$ can be triangulated into simplices. We observe a relation between the Fourier transforms and so, the Cesaro means of $C$ and its simplices. This leads to a complex integral that enables us to apply the residue theorem, the Cesaro mean can be written as an expression whose leading term is given in terms of the Reimann zeta function, which is the counting polynomial.

Namely, first let $a_1, \dots, a_d > 0$ and let $\zeta(s)$ denote the Reimann zeta function. Let 
\be
p(t) = p_{a_1, \dots, a_d}(t) = \sum_{k=0}^{d}c_k t^k
\ee
where
 \be
 c_d= \lambda(C)= \frac{2^d a_1, \dots, a_d}{d!}
 \ee
 and \be c_k= \frac{2^da_1, \dots, a_d}{(2\pi i)^{d-k} k!} \sum_{l=1}^{d} \sum_{1 \leq j_1< \dots < j_l \leq d} \sum_{ \substack{i_1+\dots+i_l=d-k \\ i_1,\dots,i_l \geq 2 \\ 2|i_1,\dots,i_l} } \frac{-2 \zeta(i_1)}{a_{j_1} ^{i_1}} \dots \frac{-2 \zeta(i_l)}{a_{j_l} ^{i_l}}.
 \ee
Secondly, let $a_1, \dots, a_d > 0$ and $p(t)$ defined as above. Then define
\be q(t)= q_{a_1, \dots, a_d}(t) = \frac{1}{2^d} \sum_{I \subseteq [d]} p_{a_i : i \in I}(t).
\ee
We can now state Borda's result for lattice points in scalar-dilations of quadratic irrational polytopes of the form $C$ and $S$.

\begin{thm}[Borda]\label{Borda} For the polytopes defined in \eqref{C} and \eqref{S}, considering any real $t>1$ and $ \epsilon>0 $, there exist computable polynomials, $p(t)$ and $q(t)$ such that
\begin{equation}
| tC \cap \mathbb{Z}^d| = p(t) + O \left(t^{\frac{(d-1)(d-2)}{ 2d-3} + \epsilon}\right)
\end{equation}
\begin{equation}
    : tS \cap \mathbb{Z}^d| = q(t) + O \left(t^{\frac{(d-1)(d-2)}{ 2d-3} + \epsilon} \right).
\end{equation}
We know that $p(t)$ and $q(t)$ are the leading terms of the Ehrhart polynomials associated with $C$ and $S$ and satisfy the reciprocity law \eqref{rec}
\end{thm}

We can now state the main result of this paper. 

\begin{thm}\label{main}
Let $P$ be a polytope in $\mathbb{R}^n$ and $D>1$ be a square free integer. Suppose every coordinate of every vertex of $P$ is in $\mathbb{Q}(\sqrt D)$. Let $\textbf{t}= (t_1, \dots , t_n)$ be an $n$ dimensional dilation vector applied to $P$ that gives $\textbf{t}P$. Let $t_{\textnormal{max}} = \max_{1\le i \le n}(t_i)$. Then
\be
|\textbf{t}P \cap \mathbb{Z}^n| = \sum_{i=1}^{2^{n+1}} q(t) + O \left(t_\textnormal{max}^{\frac{(d-1)(d-2)}{2d-3}} \right).
\ee
The leading term of each component of the polynomial $|\textbf{t}P \cap \mathbb{Z}^n|$ satisfies the reciprocity law as in \eqref{rec} and hence, is said to behave like an Ehrhart Polynomial. 
\end{thm}

Before proceeding to the proof, we first prove a lemma that will be useful for the main theorem.

\begin{lem}
Let $C$ and $S$ be defined as in \eqref{C} and \eqref{S} respectively. Then every polytope of the form $C$ can be decomposed into components expressed in the form of polytopes of the form $S$. 
 \end{lem}
\begin{proof} Let us proceed using induction on the dimension of $C$. Consider a two dimensional polytope, 
\be
C_2 = \{ x \in \mathbb{R}^2 : \dfrac{|x_1|}{a_1} + \dfrac{|x_2|}{a_2} \leq 1 \}.
\ee
 In every quadrant of the $x_1x_2$ plane, the absolute values of the $x_i$'s change signs. This makes it clear that $C$ is the union of 
\begin{align}
S_1 &= \{ x \in \mathbb{R}^2 : \dfrac{x_1}{a_1} + \dfrac{x_2}{a_2} \leq 1 \}\\
S_2 &= \{ x \in \mathbb{R}^2 : \dfrac{-x_1}{a_1} + \dfrac{x_2}{a_2} \leq 1 \}\notag\\
S_3 &= \{ x \in \mathbb{R}^2 : \dfrac{x_1}{a_1} + \dfrac{-x_2}{a_2} \leq 1 \}\notag\\
S_4 &= \{ x \in \mathbb{R}^2 : \dfrac{-x_1}{a_1} + \dfrac{-x_2}{a_2} \leq 1 \}.\notag
\end{align}

Next, assuming that the lemma holds for dimension $n$, consider an $n+1$ dimensional polytope, 
\be
C= \{ x \in \mathbb{R}^n : \dfrac{|x_1|}{a_1} + \dfrac{|x_2|}{a_2} + \dots + \dfrac{|x_n|}{a_n} \leq 1 \}.
\ee
Now since we have $n+1$ variables, and each term in $C$ can either be positive or negative, we get $2^{n+1}$ combinations of the $x_i$, all distinct. Putting into equations,
\begin{align}
\dfrac{x_1}{a_1}+ \dfrac{x_2}{a_2}+ \dots + \dfrac{x_{n+1}}{a_{n+1}} &\leq 1 \\
\dfrac{-x_1}{a_1}+ \dfrac{x_2}{a_2}+ \dots + \dfrac{x_{n+1}}{a_{n+1}} &\leq 1\notag\\
 \dfrac{x_1}{a_1}+ \dfrac{-x_2}{a_2}+ \dots + \dfrac{x_{n+1}}{a_{n+1}} &\leq 1 \notag\\
\dfrac{x_1}{a_1}+ \dfrac{x_2}{a_2}+ \dots + \dfrac{-x_{n+1}}{a_{n+1}} &\leq 1 \notag\\
\vdots\notag\\
 \dfrac{-x_1}{a_1}+ \dfrac{-x_2}{a_2}+ \dots + \dfrac{-x_{n+1}}{a_{n+1}} &\leq 1.\notag
\end{align}
Each of these combinations is geometrically,  a polytope of the type $S$ in each of the hyperoctants. Hence, the result follows.
\end{proof}

\section{Proof of Theorem \ref{main}}
\label{proof}

We want to prove that every $n$ dimensional quadratic irrational polytope has a counting polynomial that shows Ehrhart-polynomial like properties.
The proof follows by induction on the dimension $n$.

First consider a one dimensional simplex $S$, which is an interval with irrational endpoints. The lattice count is computed using the integer function. The interval under consideration is  $\left[\dfrac{1}{a_1}, \dfrac{1}{a_2} \right]$, where $a_1, a_2$ are quadratic irrationals. For $\textbf{t} = (t_1,t_2) \in \mathbb{Z}^2$, $S^{(t)}$ is given by the end points,
\be
\dfrac{t_1}{a_1} \leq x \leq \dfrac{t_1}{a_1}.
\ee
So that we obtain,
\be
j_s(\textbf{t})=\left| \left[\dfrac{t_2}{a_2} \right]- \left[\dfrac{t_1}{a_1} \right]\right|.
\ee
It is easy to see here that 
\be
j_s(-\textbf{t}) = \left|\left[\dfrac{-t_2}{a_2}\right]- \left[\dfrac{-t_1}{a_1} \right] \right| = \left|-\left( \left[\dfrac{-t_2}{a_2}\right]- \left[\dfrac{-t_1}{a_1} \right] \right)\right|= - j_s(\textbf{t}).
\ee
This proves the reciprocity law in one dimension for quadratic irrationals.

One can also consider a two dimensional simplex $C$ with vector dilations of $\textbf{t} =(t_1,t_2,t_3,t_4)$,
\be
C= \{ x \in \mathbb{R}^2 : \dfrac{|x_1|}{a_1}+ \dfrac{|x_2|}{a_2} \leq \textbf{t} \}.
\ee
The coefficient matrix obtained for this polytope in two dimension is
\be
\begin{bmatrix}
    \dfrac{1}{a_1} & \dfrac{1}{a_2}   \\
    \dfrac{-1}{a_1} & \dfrac{1}{a_2}  \\
    \dfrac{-1}{a_1} & \dfrac{-1}{a_2}  \\
    \dfrac{1}{a_1} & \dfrac{-1}{a_2}  
\end{bmatrix}
\begin{bmatrix}
    x_{1} &   \\
    x_{2} & 
\end{bmatrix}
\leq
\begin{bmatrix}
    t_{1} &   \\
    t_{2} & \\
    t_{3}  & \\
    t_{4} &
\end{bmatrix}.
\ee
This can be written as a system of inequalities
\begin{align}
    \dfrac{x_1}{a_1}+ \dfrac{x_2}{a_2} \leq t_1 \\
    \dfrac{-x_1}{a_1}+ \dfrac{x_2}{a_2} \leq t_2 \notag\\
    \dfrac{-x_1}{a_1}- \dfrac{x_2}{a_2} \leq t_3 \notag\\
    \dfrac{x_1}{a_1}- \dfrac{x_2}{a_2} \leq t_4.\notag
\end{align}
Each of these inequalities is a plane along every axis as $S_i$,
\begin{equation}
  S_1= \{ x \in \mathbb{R}^2 : \dfrac{x_1}{a_1}+ \dfrac{x_2}{a_2} \leq t_1 \}  
\end{equation}\be S_2= \{ x \in \mathbb{R}^2 : \dfrac{-x_1}{a_1}+ \dfrac{x_2}{a_2} \leq t_2 \}\notag
\ee
\be
S_3= \{ x \in \mathbb{R}^2 : \dfrac{-x_1}{a_1}+ \dfrac{-x_2}{a_2} \leq t_3 \}\notag
\ee
\be
S_4= \{ x \in \mathbb{R}^2 : \dfrac{x_1}{a_1}+ \dfrac{-x_2}{a_2} \leq t_4 \}.\notag
\ee
Each of these $S_i$'s have a separate dilation factors, as seen from the inequalities.  From Theorem \ref{Borda}, for each $i$, we have the result
\be
|t_i S_i \cap \mathbb{Z}^2|= q(t_i)+ O \left( t_i^{\frac{(d-1)(d-2)}{2d-3}} \right).
\ee

Sum this over all values of $i$ from 1 to 4 to get $C$, while considering the error term for the maximum of all $t_i$'s. The lattice points in $C$ for $\textbf{t}= (t_1,...t_4)$ are given by 
\begin{align}
 |\textbf{t}C \cap \mathbb{Z}^2|&= \sum_{i=1}^{4} |t_i S_i \cap \mathbb{Z}^2|\\
 &= \sum_{i=1}^{4} q(t_i)+ O \left( t_{\text{max}}^{\frac{(d-1)(d-2)}{2d-3}} \right)
\end{align}
where $t_{\text{max}} = \max(t_1,t_2,t_3,t_4)$. Since the leading term $q(t_i)$ for each $i$ satisfies a form of the Ehrhart-Macdonald reciprocity law, $q(-t_i)= (-1)^2 q(t_i)$, the sum of all $q$'s also satisfies the reciprocity law. This shows that $q$ is not a quasi-polynomial but has properties resembling Ehrhart polynomials that are a wider class of polynomials.
The induction hypothesis holds true for $n=2$.

Assuming that the hypothesis is true for all $n$ dimensional quadratic irrational polytopes, the counting polynomial $q$ for all such $n$ dimensional polytopes satisfied the reciprocity law and thus, showed Ehrhart polynomial-like properties. Consider the $n+1$ dimensional polytope
\be
C = \{ x \in \mathbb{R}^{n+1} : \dfrac{|x_1|}{a_1}+ \dfrac{|x_2|}{a_2}+ \dots + \dfrac{|x_{n+1}|}{a_{n+1}} \leq 1 \}
\ee
This gives the following system of inequalities,
\be
\begin{bmatrix}
    \dfrac{1}{a_1} & \dfrac{1}{a_2} & \dots & \dfrac{1}{a_{n+1}}   \\
    \dfrac{-1}{a_1} & \dfrac{1}{a_2} & \dots & \dfrac{1}{a_{n+1}}  \\
    \vdots & \vdots & \vdots & \vdots \\
    \dfrac{1}{a_1} & \dfrac{1}{a_2} & \dots & \dfrac{-1}{a_{n+1}} \\ 
\end{bmatrix}
\begin{bmatrix}
    x_{1} &   \\
    x_{2} &    \\
    \vdots & \\
    x_{n+1} &    \\
\end{bmatrix}
\leq
\begin{bmatrix}
    t_{1} &   \\
    t_{2} & \\
    \vdots & \\
    t_{2^{n+1}}  & \\
\end{bmatrix}.
\ee
Each inequality above is a single face of the polytope in each of the $2^{n+1}$ hyperoctants.

Now, lets look at each $S_i$ individually. By symmetry, it suffices to consider the positive direction of the $(n+1)$-th axis, we take the first $2^n$ components $S_i$ of $C$. 
For each $S_i$, the cross-section at $x_{n+1}=0$ is 
\be
\{ x \in \mathbb{R}^n :  \dfrac{x_1}{a_1}+ \dots + \dfrac{x_n}{a_n} \leq t_i\}
\ee
which is the corresponding $S_i$ in $n$-dimensions. For every subsequent lattice point on the $(n+1)$-th axis, 
\be
S_i= \{ x \in \mathbb{R}^n :  \dfrac{x_1}{a_1}+ \dots + \dfrac{x_n}{a_n} \leq t_i- x_{n+1}\}
\ee
 for $x_{n+1}$ running from 0 to $ \lfloor{a_{n+1}(c_1t_1+ \dots + c_nt_n)} \rfloor$. The counting polynomial then becomes 
\begin{multline}
\sum_{x_{n+1}= 0}^{ \left \lfloor{a_{n+1}(c_1t_1+ \dots +c_nt_n)}\right \rfloor}|(t_i-x_{n+1}) S_i \cap \mathbb{Z}^{n+1}|\\
=\sum_{x_{n+1}= 0}^{  \left \lfloor{a_{n+1}(c_1t_1+ \dots +c_nt_n)}\right \rfloor }  q(t_i-x_{n+1})+ O \left( (t_i-x_{n+1})^{\frac{(d-1)(d-2)}{2d-3}} \right).
\end{multline}
The $n$-dimensional polynomials, $q(t-x_{n+1})$ satisfy the reciprocity law. Therefore, 
\begin{multline}
|\textbf{t}C \cap \mathbb{Z}^{n+1}| 
\sum_{i=1}^{2^{n+1}} |(t_i-x_{n+1})S_i\cap \mathbb{Z}^{n+1}| \\
=\sum_{i=1}^{2^{n+1}}\sum_{x_{n+1}= 0}^{  \left \lfloor{a_{n+1}(c_1t_1+ \dots +c_nt_n)}\right \rfloor }  q(t_i-x_{n+1})+ O \left( (t_i-x_{n+1})^{\frac{(d-1)(d-2)}{2d-3}} \right) .
\end{multline}
By the induction hypothesis, the counting polynomial for $(n+1)$-dimensional polytope that looks like $C$ is Ehrhart-like, thus it satisfies a form of the reciprocity law.

In a more general setting, $P$ can be of any kind and not necessarily resemble $C$. Then, $P$ can be triangulated in a number of simplices that look like $C$. This just adds another layer to the proof in saying that the polynomial for $P$ can be expressed in terms of the polynomial $|\textbf{t}C \cap \mathbb{Z}^n|$ which in turn is broken down into simpler components, and is left to the reader. Therefore, Theorem \ref{main} holds for all polytopes $P$ that satisfy the hypothesis.

\end{document}